\documentclass{article}

\usepackage{amsmath}
\usepackage{amssymb}
\usepackage{amsthm}

\newcommand{\calm}{{\cal M}}

\newcommand{\calc}{{\cal C}}

\newcommand{\calt}{{\cal T}}

\newcommand{\al}{\alpha}                

\newtheorem{thm}{Theorem}[section]
\newtheorem{defn}{Definition}[section]
\newtheorem{lemma}{Lemma}[section]

\def\Th{\Theta}

\begin{document}

\title{A Fluid Limit for Processor-Sharing Queues Weighted by Functions of Remaining Amounts of Service}

\author{Yingdong Lu \\ IBM T.J. Watson Research Center \\Yorktown Heights, NY 10598, USA}

\date{September, 2019}
\maketitle

\begin{abstract}
We study a single server queue under a processor-sharing type of scheduling policy, where the weights for determining the sharing are given by functions of each job's remaining service(processing) amount, and obtain a fluid limit for the scaled measure-valued system descriptors. 
\end{abstract}



\section{Stochastic Processes and Fluid Limits}
\label{sec:one}

Consider a single server queueing model, in which jobs in the system share the processing capacity according to their remaining service amount. More specifically, at any time $t\ge 0$, if there are $N$ jobs in the system, then they will be served simultaneously, and the service capacity for the n-th job is $c_n(t) = \frac{w(R_n(t))}{\sum_{k=1}^Nw(R_k(t))}$, where $R_n(t)$ denotes the remaining service amount of the n-th job at time $t$ for $n=1,2,\ldots, N$, and $w(\cdot)$ represents a general weight function. We assume that the total capacity of the server is normalized to be one. As a weight function, $w(\cdot)$ is a bounded, continuous and differentiable function satisfying the following conditions: $w(0)=0$ and $w'(x)>0$ for all $x\ge 0$.  
This type of processor-sharing systems can be used to model many modern computing arrangements, such as server farms, see e.g. \cite{6120298}, and mainframe computers with flexible and dynamic logical partitions, see e.g. \cite{5386837}, where we see dynamic/adaptive scheduling policies are implemented real time based on the load of the systems. We assume that the exogenous arrival process $E(t)$ is a delayed renewal process with rate $\al>0$. The inter-arrival time sequence of the renewal process is denoted as $\{u_n\}, n\ge 1$. Note that $u_1$ might be different from all the others, as it is the residue inter-arrival time. $V(i)$ denotes the amount of service required by the first $i$ arrivals. $\nu_i$ are the service amount random variable for each arrival $i=1,2, \ldots$. Therefore, we can define the following related quantities.  
\begin{align*}
U_i & = \sum_{j=1}^iu_j, \quad
E(t) = \sup\{i\ge 0: U_i\le t\},\\
Z(t) &= Z(0) + E(t) -D(t), \quad
D(t)  = \sum_{j=1}^{Z(0)} {\bf 1}_{\{\nu_j \le S_j(t)\}} + \sum_{j=Z(0)+1}^{Z(0) + E(t)}  {\bf 1}_{\{\nu_j \le S_j(t)\}}.
\end{align*}
where $S_j(t)$, the cumulative service amount the $j$-th job received, satisfies the following relationship, for any $t$ and $t+h$ before the departure time of job $j$, 
\begin{align*}
S_j(t+h) = S_j(t) + \int_t^{t+h} \frac{w(R_j(s))}{\sum w(R_i(s))} ds,
\end{align*}
with $R_j=\nu_j -S_j(t)$ and $Z(t)$ denotes the number of jobs in the system. 

The object under consideration in this paper will be a sequence of this type of queueing systems indexed by $r=1,2, \dots$, and they satisfy the following heavy traffic property.  As $r\rightarrow \infty$, $\al^r \rightarrow \al $ and $\nu^r \rightarrow \nu$ weakly, and $\al \langle \chi, \nu \rangle =1$ with $\chi(x) =x$. Here the limit takes values in ${\mathcal M}_F$, the set of finite, nonnegative and Borel measures on ${\mathbb R}_+$, equipped with the topology of weak convergence. Meanwhile, for any $\xi \in {\mathcal M}_F$, and a real valued Borel measurable function $g(\cdot)$, define, $\langle g, \xi \rangle = \int_{{\mathbb R}_+} g(x) \xi(dx)$.  For system-$r$, the parameter $\rho^r=\al^r \langle \chi, \nu^r \rangle$ denotes the traffic intensity. In addition, we assume that the initial condition of the systems are similarly scaled. More specifically, denote $\mu^r(0)$ as the measure that represents the initial jobs of the $r$-th system, i.e. the descriptor of the jobs, along with their remaining service amount. We assume that, there exists a measure $\Th$ such that, $(\langle{\bar \mu}^r(0), \langle \chi, {\bar \mu}^r(0)\rangle ) \Rightarrow (\Th, \langle \chi, \Th\rangle ),  \hbox{as } n\rightarrow \infty, $
which means that the scaled initial system measures, ${\bar \mu}^r(0)= \frac{1}{r} \mu^r(0)$, converge weakly to a probability distribution. 

In this paper, we will establish the fluid limit, i.e. the functional law of large number, of such sequence of measure-valued stochastic processes. More specifically, in sec. \ref{sec:heuristic}, we will provide a heuristic argument for obtaining the specific form of the fluid limit; then in sec. \ref{sec:existence}, the existence and uniqueness of the fluid limit will be established through a Picard iteration; and finally, in sec. \ref{sec:tightness}, we will establish the convergence to the fluid limit by establishing the tightness of the scaled processes, and a study of the limit points of converging subsequences. 

It can be seen that, with a weight function $w(\cdot)$, we are dealing with a processor-sharing system with more flexibility. The introduction of the weight function also allows us to model and study the applications mentioned above.  The logic followed in this paper, establishing the existence and uniqueness of the fluid limit via tightness and convergence arguments, is the same as that in \cite{gromoll2002}. For the proof of the existence and uniqueness of the fluid limit, a rather concise Piccard iteration, rooted in the theory of ordinary differential equations, has been used. The tightness and convergence arguments in sec. \ref{sec:tightness} rely on results established in \cite{gromoll2002} or their slight variations. Therefore, some of the arguments that are same or similar to those in that paper will be pointed out with detailes omitted.

\section{A Heuristic Derivation of the Fluid Limit}
\label{sec:heuristic}

In the study of processor-sharing queue, see e.g. \cite{gromoll2002}, the process 
$\mu^r(t) =\sum_{i=1}^{E^r(t)} {\bf 1}_+(R^r_i(t)) \delta_{R^r_i(t)}$
fully captures the dynamics of the queue, 
where ${\bf 1}_+ (\cdot) ={\bf 1} \{ (0,\infty) \} (\cdot)$, and $\delta_x$ denotes the point measure on point $x$. For the stochastic processing systems considered in this paper, we will provide a heuristic derivation of the fluid limit the special case of Poisson arrivals, for the purpose of providing a logic on the form of the fluid limit will take. The existence and uniqueness of the fluid limit, as well as the convergence of the stochastic processes, will be established in the two sections that follow.  

Under the Poisson assumption, on the event that there is no new arrival in the interval of $[t, t+\Delta t]$, which will have probability $1-\al \Delta t + o(\Delta t)$, for any test function $f(\cdot)$, consider the part of the queue with jobs arrive after time $0$, we have, 
\begin{align*}
\langle f, \mu^r(t+\Delta t)\rangle- \langle f, \mu^r(t)\rangle = & \sum_{i=1}^{E^r(t)} \left\{f\left( R^r_i(t) -\frac{w(R^r_i)}{\sum_i w(R^r_i)} \Delta t+o(\Delta t)\right) -f(R^r_i) \right\} \\ = &  -\sum_{i=1}^{E^r(t)} f' (R^r_i)\frac{w(R^r_i)}{\sum_i w(R^r_i)} \Delta t+o(\Delta t).
\end{align*}
On the event that there is one arrival in the interval of $[t, t+\Delta t]$, which will have probability $\al \Delta t + o(\Delta t)$,  then, for any test function $f(\cdot)$ that are bounded and absolutely continuous, we have, similarly,
\begin{align*}
&\langle f, \mu^r(t+\Delta t)\rangle- \langle f, \mu^r(t)\rangle  = & - \sum_{i=1}^{E^r(t)} f' (R^r_i)\frac{w(R^r_i)}{\sum_i w(R^r_i)} \Delta t+ \langle f, \nu \rangle+o(\Delta t).
\end{align*}
Based upon above observation, we will expect a scaled system characteristics process eventually behave similarly in fluid limit. More specifically, for $r=1, 2, \ldots$, define the following process,
\begin{align*}
{\bar E}^r(t) =\frac{1}{r} E^r(rt), \quad {\bar Z}^r(t) =\frac{1}{r} Z^r(rt),\quad {\bar \mu}^r(t) =\frac{1}{r} \mu^r(rt), {\bar Z}^r(t) =\frac{1}{r} Z^r(rt)
\end{align*}
and we have the following result,
\begin{thm}
	\label{thm:fluid}
Under the heavy traffic condition, the sequence of scaled system characteristics, ${\bar \mu}^r(t)$ converges to a fluid limit, ${\bar \mu}(t)$, governed by the following dynamics.
\begin{align}
\label{eqn:fluid_mode_D}
\frac{d}{dt} \langle f, {\bar \mu}(t) \rangle = -\frac{\langle f' w, {\bar \mu}(t) \rangle }{\langle w, {\bar \mu}(t) \rangle } + \alpha \langle f, \nu \rangle.
\end{align}
for any $f(\cdot)\in {\mathbb C}$ where ${\mathbb C} =\{ g\in  {\bf C}_b^1({\mathbb R}_+): g(0)=g'(0) =0\}$ and $C_b^1({\mathbb R_+})$ represents the space of once continuous differential functions defined on ${\mathbb R_+}$, the set of nonnegative real numbers. 
\end{thm}
To prove this theorem, we need to address the following two questions, 1) the existence and uniqueness of the fluid process defined as above; 2) the convergence of the scaled state processes. These questions are answered in the following two sections respectively. 


\section{The Existence and Uniqueness of the Fluid Limit}
\label{sec:existence}

Express the fluid limit \eqref{eqn:fluid_mode_D} in an integral form, we have, 
\begin{align}
\label{eqn:fluid_mode}
\langle g, {\bar \mu}(t) \rangle = \langle g, {\bar \mu}(0) \rangle-\int_0^t  \frac{\langle g'w,  {\bar \mu}(s) \rangle }{\langle w,  {\bar \mu}(s) \rangle}ds + \al t \langle g, \nu \rangle.
\end{align}
To establish  the existence and uniqueness of the fluid limit, we will follow a different approach from that developed for processor sharing queue in \cite{gromoll2002}. Especially, for the existence, we will adopt the Picard iteration, a commonly used procedure for proving the existence of the solutions to ordinary differential equations. For this purpose, the iterations are required to be defined on functional spaces of signed measures.
\begin{defn}
	A set function $\nu$ defined for a metric space $X$ is called a signed measure if $\nu(\emptyset) =0$, the domain of the definition is a $\sigma$-algebra, and $\nu$ is $\sigma$-additive.
\end{defn}
Hahn-Jordan decomposition ( see e.g.  \cite{halmos1976measure}) assures that, for any signed measure $\nu$, there exist two measures $\nu^+$ and $\nu^-$ satisfying $\nu= \nu^+-\nu^-$, and $\nu^+, \nu^- \ge 0$. Furthermore,  $\nu^+$ and $\nu^-$ are orthogonal, i.e., $\nu^+(B) >0\Rightarrow \nu^-(B) =0$ and $\nu^-(B) >0\Rightarrow \nu^+(B) =0$. One can thus define the total variation norm: $||\nu|| = \nu^+(X) + \nu^-(X)$. The space of the signed measures, under the total variation norm, forms a Banach space. This space will be denoted as ${\cal M}$, and it is also known as the {\it ca}-space, see, e.g. \cite{DunfordSchwartz}. Meanwhile, if the underlying space $X$ is locally compact separable, the Riesz\--Markov\--Kakutani Representation Theorem (See, e.g. \cite{Rudin:1987:RCA:26851}) implies that ${\cal M}$ is the dual of the Banach space of all continuous real functions on $X$ under the supremum norm. Since our test functions live on ${\bf C}_b^1({\mathbb R}_+)$, in other words, there are restrictions on both the value of the functions and their derivatives, the measure space defined by duality will be slightly larger. More specifically, the norm of the measures will be defined as, $||\mu||_1= \sup\{ \langle f, \mu \rangle, ||f||_\infty\le 1, ||f'||_\infty\le 1\}$.


In our study $X$ is ${\mathbb R}$, hence the assumptions of the representation theorem are satisfied. Similar to the Picard iteration for ordinary differential equation, we will show that the following ''local'' version of the fluid model exists
\begin{align}
\label{eqn:fluid_mode_local}
\langle g, {\bar \mu}(t) \rangle = \langle g, {\bar \mu}(t_0) \rangle- \int_{t_0}^t\left[\frac{ \langle g' w,  {\bar \mu}(s) \rangle }{\langle w,  {\bar \mu}(s) \rangle} - \al  \langle g, \nu(s)\rangle \right]ds.
\end{align}
for any $t_0$, and $t\in [t_0, t_0+\Delta]$ for some $\Delta>0$ independent of $t_0$. The uniqueness is immediate, and its extension to the global version follows from the uniformality of $\Delta$, apply the same argument as those in ODE, see, e.g. \cite{coddington1955theory}.

First,  let us start with an arbitrary initial continuous map from ${\mathbb R}_+$ to {\it ca}-space ${\cal M}$, denoted as $ {\bar \mu}_0(t)$, for example, for simplicity, we can have ${\bar \mu}_0(t)={\bar \mu}_0(0)$, that is they all equal to the fluid limit of initial state. This is a member of the space $C([t_0, t_1], {\cal M})$ with the supremum norm $||\mu(t)||_c:=\sup_{t\in [t_0, t_1]} ||\mu(t)||_1$ for certain $t_1$. From the above duality, we know that, for any time $t \in [t_0, t_1]$,  $\langle f, {\bar \mu}_0(t)\rangle$ defines a functional on $\calc(X)$, i.e. the space of continuous functions under the supremum norm $||\cdot ||_\infty$. 

To update from $ {\bar \mu}_n(t)$ to $ {\bar \mu}_{n+1}(t)$, for any $n \in {\mathbb Z}_+$,  define, for any function $g\in \calc(X)$,
\begin{align}
\label{eqn:functional}
T_{n+1}(t) g  &= \langle g,  {\bar \mu}_n(t_0) \rangle - \int_{t_0}^t \left[\frac{\langle g'w, {\bar \mu}_n(s) \rangle}{ \langle w, {\bar \mu}_n(s) \rangle}  - \al  \langle g, \nu(s)\rangle \right]ds.
\end{align}
Thus, for each $t$, $T_{n+1}(t) $ is a functional on $\calc(X)$, again by the Riesz\--Markov\--Kakutani Representation Theorem, this decides a measure on $X$. Collectively, this gives us ${\bar \mu}_{n+1}(t)$. Let us denote this mapping $\calt$, it apparently does not depend on $n$. More precisely, for any $\eta(t) \in C([t_0, t_1], \calm)$, $\calt (\eta(t))$ is also a member of $C([t_0, t_1], \calm)$ defined by its action on the dual space,
\begin{align}
\label{eqn:functional_general}
\langle g,  {\bar \eta}(t_0) \rangle - \int_{t_0}^t \frac{\langle g'w, {\bar \eta}(s) \rangle}{ \langle w, {\bar \eta}(s) \rangle} ds.
\end{align}


\begin{lemma}
	For a given $\delta>0$, there exists a $\Delta>0$ that is independent of $t_0$, and for any $t\in [t_0. t_0+\Delta]$, let $\mu(t)$ and $\eta(t)$ in ${\mathbb R}_+\rightarrow {\cal M}$ $\langle w, \mu\rangle \ge \delta$, there exist $\zeta \in(0,1)$ such that, for any test function $g$, we have, 
	\begin{align}
	\label{eqn:contraction}
	|| \calt \mu(t) - \calt \eta(t)||_c \le\zeta ||\mu(t)-\eta(t)||_c.
	\end{align}
\end{lemma}
\begin{proof}
By the definition given in \eqref{eqn:functional}, we know that for any $g(\cdot)$ satisfies $||g||_\infty\le1$ and $||g'||_\infty\le 1$, 
\begin{align*}
|\langle  \calt \mu(t), g \rangle - \langle \calt \eta(t), g \rangle| = \Big|  \int_{t_0}^t \left[\frac{\langle g'w,\mu(s) \rangle}{\langle w, \mu(s)\rangle}- \frac{\langle g'w,\eta(s) \rangle}{\langle w, \eta(s)\rangle} \right]ds \Big|.
\end{align*}
The integrand on the right hand side can be written as,
\begin{align*}
 & \frac{\langle g'w,\mu(s) \rangle}{\langle w, \mu(s) \rangle}- \frac{\langle g'w,\eta(s)  \rangle}{\langle w, \eta(s) \rangle} \\ =&    \frac{\langle g'w,\mu(s)  \rangle-\langle g'w,\eta(s)  \rangle}{\langle w , \mu(s) \rangle} + \langle g'w, \eta(s)  \rangle \left[ \frac{1}{\langle w, \mu(s)  \rangle} - \frac{1}{\langle w, \eta(s) \rangle}\right].
\end{align*}
Therefore,
\begin{align*}
&|\langle  \calt \mu(t), g \rangle - \langle \calt \eta(t), g \rangle|
\\ \le &  \int_{t_0}^t \Big|  \frac{\langle g'w,\mu(s)  \rangle-\langle g'w,\eta(s)  \rangle}{\langle w , \mu(s) \rangle} \Big|ds + \int_{t_0}^t \Big|\langle g'w, \eta(s)  \rangle \left[ \frac{1}{\langle w, \mu(s)  \rangle} - \frac{1}{\langle w, \eta(s) \rangle}\right]\Big|ds
\end{align*}
The first term,
\begin{align*}
 \int_{t_0}^t \Big|  \frac{\langle g'w,\mu(s)  \rangle-\langle g'w,\eta(s)  \rangle}{\langle w , \mu(s) \rangle} \Big|ds& \le  \frac{||g'||_\infty|| w||_\infty || \mu-\eta||_c }{\delta} (t-t_0),
\end{align*}
from the definition of the $|| \cdot ||_c$ norm and the fact that $\langle w , \mu(s) \rangle \ge \delta$.
Similarly, for the second term, we have,
\begin{align*}
& \int_{t_0}^t \Big|\langle g'w, \eta(s)  \rangle \left[ \frac{1}{\langle w, \mu(s)  \rangle} - \frac{1}{\langle w, \eta(s) \rangle}\right]\Big|ds \\ \le & 
\int_{t_0}^t \Big|\langle g'w, \eta(s) \rangle \Big| \Big|  \left[ \frac{1}{\langle w, \mu(s)  \rangle} - \frac{1}{\langle w, \eta(s) \rangle}\right]\Big|ds \\ 
\le & 
\int_{t_0}^t \Big|\langle w, \eta(s) \rangle \Big| \Big|  \left[ \frac{\langle w, \eta(s) \rangle -\langle w, \mu(s) \rangle}{\langle w, \mu(s)  \rangle \langle w, \eta(s) \rangle} \right]\Big|ds \\ \le &  \frac{|| w||_\infty || \mu-\eta||_c }{\delta} (t-t_0) .
\end{align*}
Therefore for $t\le t_0+\Delta$, with $\Delta =\frac{\al \delta }{2  || w||_\infty}$, \eqref{eqn:contraction} holds. It is easy to see that $\Delta$ does not depend on $t_0$.  \end{proof}

The above lemma guarantees that the mapping from $ {\bar \mu}_n$ to $ {\bar \mu}_{n+1}$ is a contraction for a neighborhood of $t_0$. Thus, we can apply the Banach fixed point theorem to establish the local existence of the fluid model for a neighborhood of each point, following the same procedures for differential equation, see, e.g. \cite{coddington1955theory}. Similarly, the uniformality of $\Delta$ will allow the existence to be extended to any compact sets. Note that the restriction of $\langle w, \mu\rangle>\delta$ is only a technical assumption, as pointed out in \cite{gromoll2002}, the fluid limit processes usually have constant total workload. It is also easy to verify that the limit is nonnegative and unique.




\section{Convergence to the Fluid Limit}
\label{sec:tightness}

In this section, we will provide main arguments for the convergence of the scaled processes to the fluid limit. Much of the analysis follows from \cite{gromoll2002}, so we will focus on pointing out some of the modifications needed. 

Recall that the fluid scaling systems indexed by integers $r=1,2,\ldots$, and the performance descriptor $\mu^{r}(t)$. More precisely,
${\bar \mu}^r(t) = \frac{1}{r}\mu^{r}(rt),$
with the primitive processes of the systems satisfy, $\al^r \rightarrow \al$, $\nu^r  \rightarrow \nu, weakly,$ $E[u_{1,r}]/r  \rightarrow 0$, $E[u_{2,r}; u_{2,r}>r] \rightarrow 0$, as $r\rightarrow \infty$. The key components of the proof of the convergence of ${\bar \mu}_r(t)$ to the fluid limit \eqref{eqn:fluid_mode} are to show its tightness, which guarantees that each subsequence will have a converging subsequence, and each converging subsequence will have the same limit that is the fluid limit defined in Theorem \ref{thm:fluid}.  

The tightness of a stochastic process is defined as the tightness of the measures it induced, which means the existence of a compact set such that the measure outside which is uniformly small, and details can be found in \cite{billingsley1968convergence}.  Meanwhile, by Jakubowski's criterion, see, e.g.  \cite{gromoll2002}, the proof of the tightness is reduced to the verification of the following condtions,
\begin{itemize}
	\item
	For each $T>0$, and $0<\eta<1$, there is a compact subset $C_{T, \eta}$ of ${\mathcal M}_F$ such that 
	\begin{align*}
	\liminf\limits_{r\rightarrow \infty} P[ {\bar \mu}^r(t) \in C_{T,\eta}, \forall t \in[0,T]] \ge 1-\eta.
	\end{align*}
      \item
	For each $g\in C_b^1({\mathbb R}_+)$, the sequence of real valued processes $\{\langle g, {\bar \mu}^r (\cdot)\rangle\}$ is tight.
\end{itemize}
Essentially, we need to prove the following results,
\begin{lemma}
	\label{lem:tightness}
	Let $g\in {\bf C}_b^1({\mathbb R}_+)$, $T>0$, and $0<\beta<1$, $\eta <1$. Then there exist $M_T, \delta >0$, and $r_0>0$, such that, for any $r> r_0$, we have,
	\begin{align}
	P\left[\sup_{t\in [0,T]}|\langle g, {\bar \mu}^r (t) \rangle| \le M \right] & \ge 1-\eta, \label{eqn:tight_1}\\   P\left[\sup_{\substack{t\in [0, T-\delta]\\ h\in [0,\delta]}} |\langle g, {\bar \mu}^r (t+h) \rangle -\langle g, {\bar \mu}^r(t) \rangle| \le \beta\right] &\ge 1-\eta\label{eqn:tight_2}
	\end{align}  
\end{lemma}
\begin{lemma}
	\label{lem:limiting_point}
	All the converging subsequences of ${\bar \mu}_r (t)$ converge to the same limit governed by the dynamics in \eqref{eqn:fluid_mode_D}.
\end{lemma}
\noindent
\begin{proof}[Proof of Lemma \ref{lem:tightness}]
	Under the heavy traffic condition, there exists a set $B^r$, and real number $\epsilon$, $\gamma$, $M_0$ and $M'$, such that, on $B^r$,
	\begin{itemize}
		\item[(I)] Arrival process ${\bar E}^r(t)= \frac{1}{r}E(rt)$ satisfies, $\sup_{t\in[0, T-\ell]} |{\bar E}^r (t+\ell)- {\bar E}^r (t)|\le \epsilon.$ 
		\item[(II)] Workload process $|\langle \chi, {\bar \mu}^r(t)\rangle$ satisfies, \;
	$\sup_{t\in[0,T]}|\langle \chi, {\bar \mu}^r(t) \rangle -\langle \chi, {\bar \mu}^r(0) \rangle| < \frac{\gamma}{4} $.
        \item[(III)]
        The remaining service is cut-off at certain level, i.e., there exists a $M_0>0$, such that, $R_i(s)\le M_0$ for any job $i$ in the system. 
        \item[(IV)] The total service can be lower bounded, i.e. there exists a $M'>0$, such that,  $\langle \chi, {\bar \mu}^r(0) \rangle \ge M'$. 
	\end{itemize}
And $P[B^r]\ge 1-\eta$ for sufficiently large $r$. These results, among others that more specific to processor-sharing policy,  are established in Lemma 5.2 in \cite{gromoll2002}, and (I) to (IV) collected here are among those that are of policy independent.  More specifically, it says that with high probability, the oscillation of the scaled process can be controlled, as seen in (I),   the convergence of the workload process leads to its bounded (II), the individual service amount can be upper bounded the total workload can be also lower bounded (III), and the total workload can be also lower bounded, as seen in (IV).

To see the main result is true, we know that $\sup_{t\in[0,T]} \langle 1, {\bar \mu}^r(t) \rangle \le \langle 1, {\bar \mu}^r(0) \rangle + {\bar E}^r(T)$.
Hence we can have the following uniformly bounded, i.e. $\sup_{t\in[0,T]} ( \langle 1, {\bar \mu}^r(t) \rangle \wedge \langle \chi, {\bar \mu}^r(t) \rangle ) \le M_T$, which settles \eqref{eqn:tight_1}. 
From the scaled dynamics,
  \begin{align*}    \langle g,  {\bar \mu}^r(t+h) \rangle = \langle ({\bf 1}_{(0,\infty)} g) (\cdot - {\bar S}^i_{r;t, t+h}),  {\bar \mu}^r(t)\rangle + \sum_{i=E^r(t) +1}^{E^r(t+h)}{\bf 1}_{(0,\infty)}g(v_i^r - {\bar S}^r_{U_i^r, t+h})
  \end{align*}
 with $v_i^r$ refers to the service time of $i$-th arrival in system $r$ and the scaled version of service, ${\bar S}^i_{r;t, t+h}$ represents the amount of service received by job $i$ during time interval $[rt, rt+rh)$. We know that, it is bounded from above by $\frac{w(R^r_i(rt))}{\langle w, \mu^r(rt)\rangle}\times rh$. To see this, if there are no new arrivals, the proportion will remain to be $\frac{w(R^r_i(rt))}{\langle w, \mu^r(rt)\rangle}\times rh$; if there are new arrivals, the proportion will only be smaller. Meanwhile, we know that on $B^r$, we have, $ R_i(rt) \le M_0$ from (3), which bound the numerator from above. Then, from (2) and (4), we know that as long as $\delta$ is picked to satisfies $\delta M_0 M_T||g'||_\infty<\beta$,  this bounds the denominator from below, then the desired result follows. 
\end{proof}

\begin{proof}[Proof of Lemma \ref{lem:limiting_point}]
	We need to establish that, the limit point of each converging subsequence of  ${\bar \mu}_r(t)$ coincides with the fluid limit. 
	We will first establish the results on a subset of test functions, which has continuous and bounded derivative. Then by the fact that it is dense in the original space of continuous and bounded functions we are studying, and extend the results in the similar fashion as demonstrated in Section 5.3 in \cite{gromoll2002}. Again, the most important identity is estimation of the increment for the scaled measure, i.e. for small $h>0$, 
	\begin{align}
	&\langle g, {\bar \mu}^r(t+h) \rangle -\langle g, {\bar \mu}^r(t) \rangle\nonumber \\= &\left[\langle g, {\bar \mu}^r(t+h) \rangle - \sum_{i=1}^{ \langle 1, {\bar \mu}^r((t) \rangle}g(y - {\bar S}^i_{r;t, t+h})\right]+\left[\sum_{i=1}^{ \langle 1, {\bar \mu}^r((t) \rangle}g(y - {\bar S}^i_{r;t, t+h})- g(y)\right] \nonumber \\= & \left[\frac{1}{r} \sum_{i={\bar E}^r(t)+1}^{{\bar E}^r(t+h)}g(v^r_i - {\bar S}^{i,*}_{r;t, t+h})\right]+\sum_{i=1}^{ \langle 1, {\bar \mu}(t) \rangle}|{\bar S}^i_{r;t, t+h}g'(y^*)|. \label{eqn:difference}
	\end{align}
	for some ${\bar S}^{i,*}_{r;t, t+h}$ denoting the service amount received by new arrivals and $y_i*\in [y - {\bar S}^i_{r;t, t+h}, y]$, which is due to the mean value theorem.  Note that the first term in \eqref{eqn:difference} characterizes the changes in the remaining service time for the jobs that arrives between $rt$ and $r(t+h)$; and the second term characterizes the service changes for the jobs that are already in the system at time $rt$. Now, let us examine the first term.  We know that, $g(v^r_i - {\bar S}^{i,*}_{r;t, t+h})= g(v^r_i) - g'(v_i^*) {\bar S}^{i,*}_{r;t, t+h}$ for some $v_i^* \in [ v^r_i - {\bar S}^{i,*}_{r;t, t+h}, v^r_i]$ due to the mean value theorem. In combined with the limiting behavior of $E^r(t)$, we know that, 
\begin{align*}
\frac{1}{r} \sum_{i={\bar E}^r(t)+1}^{{\bar E}^r(t+h)} g(v^r_i) 
\end{align*}	
converges to $\al \langle g, \nu\rangle h$ as $r\rightarrow \infty$ and for $h$ small enough. Meanwhile,
\begin{align*}
\frac{1}{r} \sum_{i={\bar E}^r(t)+1}^{{\bar E}^r(t+h)} g'(v_i^*) {\bar S}^{i,*}_{r;t, t+h}
\end{align*}	
will converge to an term of order $o(h)$ since ${\bar S}^{i,*}_{r;t, t+h}<h$.

	Next, we have
	\begin{align*}
	&\sum_{i=1}^{ \langle 1, {\bar \mu}(t) \rangle}{\bar S}^i_{r;t, t+h}g'(y_i^*)\\ = & \sum_{i=1}^{ \langle 1, {\bar \mu}(t) \rangle}{\bar S}^i_{r;t, t+h}g'(x) + \sum_{i=1}^{ \langle 1, {\bar \mu}(t) \rangle}{\bar S}^i_{r;t, t+h}[g'(x) -g'(y_i^*)].
	\end{align*}
	We know that the first term, due to the dynamics of the service quantity ${\bar S}^i_{r;t, t+h}$, will converge to the quantity of $- \frac{\langle g' w, \mu^*(t)\rangle}{\langle w, \mu^*(t)\rangle} h$ where  $ \mu^*(t)$ will be the limiting point, as $r\rightarrow \infty$. The second term will converge zero due to the continuity and boundedness of $g'(\cdot)$.
	Thus, we have established that the differential equation to which the subsequence will converge is the same one that the fluid limit will satisfy. Then by the well known results in differential equation on the convergence of the solutions, see e.g. Theorem 4.1 in \cite{coddington1955theory}, we know that the solutions also converge the fluid limit. 
\end{proof}

\section{Conclusions}
\label{sec:conclusions}

In this paper, we derive a fluid limit for a class of processor sharing policies for a single server queue, motivated by applications that enabled flexible and workload-dependent scheduling among different jobs. The approach we took borrows heavily from the analysis of ordinary differential equations. We also would like to remark that, for properly selected test function, the fluid limit evolution equation can be readily solved numerically through numerical methods, such as Runge-Kutta or colocation methods, see,e.g. \cite{hairer2002geometric}.

\bibliography{Lu}{}
\bibliographystyle{abbrv}

\end{document}